\documentclass[a4paper,draft,reqno]{amsart}
\usepackage[a4paper,margin=25mm]{geometry}
\usepackage[utf8]{inputenc}
\usepackage{mathtools}
\usepackage{amsfonts}
\usepackage{amsthm}
\usepackage{amsmath}
\usepackage{amssymb}
\usepackage{tikz}
\usepackage{hyperref}
\usepackage{authblk}

\usepackage{mathtools}

\newtheorem{Theorem}{Theorem}
\newtheorem{Proposition}{Proposition}

\newtheorem{Corollary}{Corollary}
\newtheorem{Algorithm}{Algorithm}
\newtheorem{Definition}{Definition}
\newtheorem{Example}{Example}

\def\eatspace#1{#1}
\def\step#1#2{\par\kern1pt\dimen44=#2em\advance\dimen44 1.67em\hangindent\dimen44\hangafter=1\noindent\rlap{\small#1}\kern\dimen44\relax\eatspace}
\let\set\mathbb
\def\<#1>{\langle#1\rangle}

\makeatletter
\def\testb#1{\testb@i#1,,\@nil}%
\def\testb@i#1,#2,#3\@nil{%
  \draw[->] (O) --++(#1);
  \ifx\relax#2\relax\else\testb@i#2,#3\@nil\fi}
\makeatother

\usepackage[backend=bibtex,firstinits=true]{biblatex}

\title[]{The Orbit-Sum Method\\ for Higher Order Equations}


\usepackage{multicol}
\begin{document}

\maketitle

\begin{multicols}{2}
\begin{center}
\hspace{2cm}
\begin{tabular}{@{}c@{}}
    Manfred Buchacher\\
    Institute for Algebra\\ 
Johannes Kepler University Linz\\
    \normalsize manfredi.buchacher@gmail.com
  \end{tabular}%
  \end{center}
  \columnbreak
\begin{center}
\hspace{-2cm}
\begin{tabular}{@{}c@{}}
Manuel Kauers\\
Institute for Algebra\\
Johannes Kepler University Linz\\
    \normalsize manuel.kauers@jku.at
  \end{tabular}
  \end{center}
\end{multicols}

\begin{abstract}
The orbit-sum method is an algebraic version of the reflection-principle that was introduced by Bousquet-M\'{e}lou and Mishna to solve functional equations that arise in the enumeration of lattice walks with small steps restricted to $\mathbb{N}^2$. It proceeds by computing a set of algebraic substitutions that can be applied to a given functional equation, forming a linear combination of its transformed versions to the end of eliminating some of the unknowns, and eliminating further unknowns by discarding terms with negative powers. The extension of the orbit-sum method to walks with large steps was started by Bostan, Bousquet-M\'{e}lou and Melczer. They presented an algorithm that computes the minimal polynomials of the algebraic substitutions. We continue their work by explaining, among other things, how to perform computations in their splitting field on the level of ``formal'' algebraic extensions and how its elements can be interpreted as series. We thereby make use of the primitive element theorem, Gr\"{o}bner bases and the shape lemma, and the Newton-Puiseux algorithm. 
\end{abstract}

\section{Introduction}

Many generating functions can be described as solutions of certain functional equations.
One important type of such functional equations is the class of \textit{discrete differential equations} (DDE's). 
They arise in the context of the enumeration of lattice walks restricted to cones, whose systematic study was initiated in~\cite{marni,smallSteps}. 

Discrete differential equations are equations of the form
\begin{equation}\label{eq:DDE}
 F = P(x,y) + t Q(x,y,t,\Delta_x^k \Delta_y^{l} F: k,l\in\mathbb{N})
\end{equation}
where $F\in\mathbb{C}[x,y][[t]]$, $P\in\mathbb{C}[x,y]$ and $Q\in\mathbb{C}[x,y,t,v_{kl}: k,l\in\mathbb{N}]$. The operator $\Delta_x$ is the \textit{discrete derivative} with respect to $x$. It acts on $\mathbb{C}[x,y][[t]]$ by
\begin{equation*}
F(x,y;t) \mapsto \frac{F(x,y;t)-F(0,y;t)}{x}.
\end{equation*}
The operator $\Delta_y$ is the discrete derivative with respect to $y$ and defined analogously.
A DDE is a \textit{partial discrete differential equation} (PDDE) if it involves discrete derivatives of $F$ with respect to both $x$ and $y$, and an \textit{ordinary discrete differential equation} (ODDE) otherwise. The \textit{degree} of the equation is the total degree of $Q$ with respect to the $v_{kl}$'s. If its degree is at most $1$, it is a \textit{linear} DDE, otherwise it is \textit{non-linear}. If~$k+l$ is maximal among the discrete derivatives $\Delta_x^k\Delta_y^l F$ appearing in equation~(\ref{eq:DDE}), then $k+l$ is the \textit{order} of the equation.    

Any DDE has a unique solution $F\in\mathbb{C}[x,y][[t]]$ -- just consider the recurrence relation for $[t^n] F$ that results from extracting the coefficient of $t^n$. To solve a DDE means to decide whether its solution is \textit{algebraic}, \textit{D-finite} or \textit{D-algebraic}, and in case it is, to determine a polynomial or differential equation it satisfies.

There is a family of methods for solving DDE's~\cite{bousquet2006polynomial, bostan2022algorithms,bostan2023fast,buchacher2018inhomogeneous,notarantonio2022effective, smallSteps, bousquet2016elementary, bostan2010complete, bostan2017hypergeometric, mishna2009two,melczer2014singularity, large, raschel2020counting, bousquet2016square, bousquet2023enumeration, bousquet2024walks, bonnet2024galoisian, buchacher2020separating, buchacher2024separating, buchacher2024separated} that involve only operations such as
\begin{equation*}
+, \quad  \cdot,\quad \circ \qquad \text{and} \qquad [x^\geq] \quad \text{and} \quad [y^\geq],
\end{equation*}
that is, the addition, multiplication and composition of series, and the operation of discarding all terms of a series which involve negative powers in $x$ and $y$, respectively. The orbit-sum method~\cite{smallSteps,large} is one of these methods. It is used to solve linear DDE's, and proceeds in three steps. In the first step, a set of substitutions that can be applied to the given functional equation, the so-called \textit{orbit}, is determined.
In the second step, a linear combination of the various transformed versions of the functional equation
is formed to the end of eliminating all the evaluations of $F(x,0)$ and $F(0,y)$.
The resulting equation(s) then contain only the unknown series $F$ and various series obtained from the substitutions. In the third step of the method, by means of discarding terms with negative powers in $x$ and $y$, respectively, an expression for the unknown series $F$ is obtained.

\begin{Example}\label{example:simple}
We solve the equation
\begin{equation}\label{eq:DDE2}
 F = 1 + t (x+y) F + t \Delta_x F + t \Delta_y F
\end{equation}
for $F\in\mathbb{C}[x,y][[t]]$. It is equivalent to
\begin{equation}\label{eq:simple}
xy(1 - t S) F(x,y) = xy - tx F(x,0) - ty F(0,y),
\end{equation}
where $S := x+y+\bar{x}+\bar{y}$, and $\bar{x} := 1/x$ and $\bar{y}:=1/y$, and $F(x,y)\equiv F(x,y;t)$. We exploit the symmetry of $S$ and the fact that the unknowns on the right of the equation either do not depend on $x$ or on $y$. By iteratively performing the substitutions $x\mapsto \bar{x}$ and $y\mapsto \bar{y}$ that leave $S$ invariant we can derive three additional equations,
 \begin{align*}
  \bar{x} y (1-t S) F(\bar{x},y) &= \bar{x}y - t \bar{x} F(\bar{x};0) - ty F(0,y),\\
  \bar{x} \bar{y} (1-t S) F(\bar{x},\bar{y}) &= \bar{x}\bar{y} - t \bar{x} F(\bar{x};0) - t\bar{y} F(0,\bar{y}),\\
  x \bar{y}(1-t S) F(x,\bar{y}) &= x\bar{y} - t x F(x;0) - t\bar{y} F(0,\bar{y}),
 \end{align*}
which, together with equation~(\ref{eq:simple}), can be linearly combined to
 \begin{equation*}
  F(x,y) - \bar{x}^2 F(\bar{x},y) + \bar{x}^2\bar{y}^2 F(\bar{x},\bar{y}) - \bar{y}^2 F(x,\bar{y}) = 
  \frac{1 -\bar{x}^2 + \bar{x}^2\bar{y}^2 - \bar{y}^2}{1-t S}.
 \end{equation*}
Since $F(x,y)$ involves only non-negative powers of $x$ and $y$, and because all the other terms on the left-hand side of this equation in $\mathbb{C}[x,y,\bar{x},\bar{y}][[t]]$ involve a negative power of $x$ or a negative power in $y$, we find that
 \begin{equation}\label{eq:pos}
   F(x,y) = [x^\geq y^\geq] \frac{1 -\bar{x}^2 + \bar{x}^2\bar{y}^2 - \bar{y}^2}{1-t S}.
 \end{equation}
As rational functions are D-finite, and the class of D-finite functions is closed under applying $[x^\geq y^\geq]$~\cite{lipshitz}, it follows that $F$ is D-finite too.
\end{Example}

For linear partial discrete differential equations of higher order an algorithm for determining the orbit was presented in~\cite[Sec 3]{large}.
The substitutions determined by this algorithm are algebraic functions given by their minimal polynomials. 
Algebraic functions cause difficulties in the second and third step of the orbit sum method.
Based on~\cite[Sec~6.1]{buchacher2021algorithms}, and in continuation of~\cite{Rika}, we discuss these difficulties here.
In order to carry out the second step algorithmically (Sec~\ref{sec:orbit}), we need to construct an
algebraic function field that contains all the algebraic functions appearing in the orbit. This
step can be done on the level of ``formal'' algebraic extensions.
The third step however crucially depends on series interpretations of the algebraic functions,
so in order to carry it out algorithmically (Sec~\ref{sec:posPart}), we need to embed the algebraic
function field into suitably chosen fields of series. The question is then whether for the equation
at hand there exists an embedding that allows the orbit sum method to conclude. To answer this
question, we offer an algorithmic sufficient condition.

\section{Orbits, Orbit Equations, and the Orbit-Sum}\label{sec:orbit}

The substitutions we used to solve equation~\eqref{eq:DDE2} had the following property: for every substitution $(x',y')$, there were other substitutions $(x'',y'')$ and $(x''',y''')$ such that $x' = x''$ and $S(x',y') = S(x'',y'')$, and $y' = y'''$ and $S(x',y') = S(x''',y''')$. They allowed us to modify equation~\eqref{eq:DDE2} without altering~$S(x,y)$, and without altering one of the unknown evaluations of $F(x,0)$ and $F(0,y)$. These observations give rise to the notion of the $\textit{orbit}$ of a polynomial~\cite[Def~1]{large}. 
\begin{Definition}
Let $p\in\mathbb{C}[x,y,\bar{x},\bar{y}]$, and let $\sim$ be the smallest equivalence relation on~$\overline{\mathbb{C}(x,y)}^2$ such that
\begin{equation*}
(u_1,u_2)\sim (v_1,v_2)\quad \text{whenever} \quad u_1 = v_1 \text{ or } u_2 = v_2, \text{ and } p(u_1,u_2) = p(v_1,v_2).
\end{equation*}
The equivalence class of $(x,y)$ is called the orbit of $p$.
\end{Definition}

The elements of an orbit are pairs of algebraic functions which can be represented by their minimal polynomials. A (semi-)algorithm that determines them was presented in~\cite[Sec 3.2]{large}. It takes as input a Laurent polynomial and outputs, if the orbit is finite, the (pairs of) minimal polynomials of (the coordinates of) its elements. If the orbit is finite, then the splitting field of these minimal polynomials is a finite field extension of $\mathbb{C}(x,y)$. Using a constructive version of the primitive element theorem, we can perform computations in this field.
\begin{Theorem} (Primitive Element Theorem)
Let $K$ be a field of characteristic $0$, and let $L/K$ be a finite field extension. Then there is an $\alpha \in L$ such that $L=K(\alpha)$. If $m(X)\in K[X]$ is the minimal polynomial of $\alpha$, then
\begin{equation*}
L \cong K[X]/\langle m(X) \rangle.
\end{equation*}
\end{Theorem}

Given the minimal polynomial $m(X)$ of a primitive element $\alpha$ of the splitting field of a set of polynomials $m_1(X),\dots,m_n(X)$ over $\mathbb{C}(x,y)$, computations just amount to polynomial arithmetic in $\mathbb{C}(x,y)[X]/\langle m(X) \rangle$, that is, adding and multiplying polynomials over $\mathbb{C}(x,y)$, performing division with remainder and computing modular inverses using the extended Euclidean algorithm. It remains to clarify how the minimal polynomial of a primitive element can be found, and how elements of the splitting field can be expressed in terms of the primitive element. Gr{\"o}bner bases and the shape lemma~\cite[Thm~3.7.25]{kreuzer2000} provide an answer.

\begin{Definition}
Let $I\subseteq K[x_1,\dots,x_n]$ be a zero-dimensional ideal. It is said to be in normal $x_i$-position, $i\in\{1,\dots,n\}$, if any two zeros $(a_1,\dots,a_n)$ and~$(b_1,\dots,b_n)$ of $I$ in $\overline{K}^n$ satisfy~$a_i\neq b_i$.
\end{Definition}

\begin{Theorem}\label{theorem:prime} (Shape Lemma)
Let $K$ be a field of characteristic $0$, and let $I\subseteq K[x_1,\dots,x_n]$ be a $0$-dimensional radical ideal in normal $x_n$-position. Then $I$ has a Gr{\"o}bner basis with respect to lex order which is of the form
\begin{equation*}
\{x_1-g_1,\dots ,x_{n-1}-g_{n-1},g_n\}
\end{equation*}
 for some $g_1,\dots, g_n \in K[x_n]$. In particular, the set $\mathrm{Z}(I)$ of zeros of $I$ is 
 \begin{equation*}
 \mathrm{Z}(I) = \{ (g_1(a),\dots,g_{n-1}(a),a)\in K^n : g_n(a) = 0 \}.
 \end{equation*}
\end{Theorem}

Assume that $m_1(X),\dots, m_n(X)\in\mathbb{C}(x,y)[X]$ are irreducible and pairwise distinct, and let $I$ be the ideal generated by $m_{i}(X_{ij})$ and $1- Y \prod_{ij\neq kl} (X_{ij}-X_{kl})$ and $ Z - \sum_{ij} a_{ij}X_{ij}$, where the $X_{ij}$'s and $Y$ and $Z$ are variables and $a_{ij}\in\mathbb{Q}$, for $i=1,\dots,n$ and $j=1,\dots,\deg_X(m_i)$. Without loss of generality we assume that all the assumptions of the shape lemma are satisfied as we can choose the $a_{ij}$'s such that $I$ is in normal $Z$-position~\cite[Def~3.7.21]{kreuzer2000} and replace $I$ by its radical $\sqrt{I}$ without altering the set of its zeros~\cite[Cor~3.7.16]{kreuzer2000}. The shape lemma implies that the Gr{\"o}bner basis of $I$ gives rise to a polynomial $m(X)\in\mathbb{C}(x,y)[X]$ whose roots $\alpha$ are primitive elements of the splitting field of $\{m_1(X),\dots, m_n(X)\}$ over $\mathbb{C}(x,y)$ and polynomials $p_{ij}(X)\in\mathbb{C}(x,y)[X]$ such that the roots of $m_1(X),\dots, m_n(X)$ are given by the $p_{ij}(\alpha)$'s. 

Let $P\in\mathbb{C}[x,y]$ and $Q_{kl}\in\mathbb{C}[x,y]$ be polynomials, and let
\begin{equation}\label{eq:kernelEquation}
F = P(x,y) + t \sum_{k,l} Q_{kl}(x,y) \Delta_x^k \Delta_y^l F
\end{equation}
be a linear discrete differential equation for $F\in\mathbb{C}[x,y][[t]]$. The \textit{kernel polynomial} of the equation is the Laurent polynomial that appears as the coefficient of $F(x,y)$ when all terms involving it are collected on the left hand side of the equation. The orbit of the equation is the orbit of its kernel polynomial. If it is finite, we can now assume that there is some $\alpha\in\overline{\mathbb{C}(x,y)}$ such that its elements are of the form $(p_1(\alpha),p_2(\alpha))$ and given in terms of $p_1(X),p_2(X)\in\mathbb{C}(x,y)[X]$ and the minimal polynomial~$m(X)\in\mathbb{C}(x,y)[X]$ of $\alpha$. The orbit equations result from replacing $(x,y)$ in equation~\eqref{eq:kernelEquation} by the elements of the orbit, and an orbit-sum is any $\mathbb{C}(x,y)[\alpha]$-linear combination of the orbit equations that does not involve any of the sections $F(\cdot,0)$ and $F(0,\cdot)$. Computing a basis of the vector space of such equations amounts to making an ansatz with undetermined coefficients for the linear combination, setting the coefficients of the sections equal to zero, and solving a system of linear equations over the field $\mathbb{C}(x,y)[X]/\langle m(X) \rangle$.

\section{Positive-Part-Extraction}\label{sec:posPart}

In the previous section we recalled how the minimal polynomials of the algebraic substitutions required by the orbit-sum method are determined and explained how Gr\"{o}bner bases and the shape lemma allow one to reduce computations in their splitting field to polynomial arithmetic. As a consequence the first two steps of the orbit-sum method can be performed algorithmically, the result of the computations being a basis of the vector space of section-free orbit equations whose elements are of the form 
\begin{equation}\label{eq:orbitEquation}
F(x,y) + \sum_{(p_1,p_2,p_3)} p_3(\alpha) F(p_1(\alpha),p_2(\alpha)) = p(\alpha),
\end{equation}
where $F\in\mathbb{C}[x,y][[t]]$ is unknown, $\alpha$ is an element of $\overline{\mathbb{C}(x,y)}$, given by its minimal polynomial over $\mathbb{C}[x,y]$, and~$p_1(\alpha), p_2(\alpha)$ and $p_3(\alpha)$ are polynomials in $\alpha$ over $\mathbb{C}(x,y)$, and $p(\alpha)$ is a polynomial in $\alpha$ over $\mathbb{C}(x,y,t)$. The purpose of this section is to give a meaning to 
\begin{equation}\label{eq:posPart}
[x^\geq y^\geq] p_3(\alpha) F(p_1(\alpha),p_2(\alpha)),
\end{equation}
and to present a sufficient and effective condition for equation~\eqref{eq:orbitEquation} to imply that 
\begin{equation*}
F(x,y) = [x^\geq y^\geq] p(\alpha).
\end{equation*}
This requires to interpret elements of $\overline{\mathbb{C}(x,y)}$ as series in $x$ and $y$. The non-negative part is then the series which results from discarding all terms which involve a negative power of $x$ or $y$, respectively. We did not stress this point in Example~\ref{example:simple} because the right-hand side of equation~\eqref{eq:pos} can unambiguously be understood as an element of~$\mathbb{C}[x,y,\bar{x},\bar{y}][[t]]$ whose non-negative part with respect to $x$ and $y$ is well-defined. In general, however, more care is necessary.

\begin{Example}
It is ambiguous to speak of the non-negative part of the series solution $Y$ of 
\begin{equation*} 
(1-x)Y-1=0.
\end{equation*}
The solution of the equation depends on the field of Laurent series over which it is solved. While in $\mathbb{C}((x))$ it is $Y = \sum_{k=0}^\infty x^k$, in $\mathbb{C}((\bar{x}))$ it is $Y = -\sum_{k=1}^\infty \bar{x}^k$, and depending on which of them we choose, we have $[x^\geq] Y = Y$ or $[x^\geq] Y = 0$.
\end{Example}

To give a meaning to expression~\eqref{eq:posPart} we embed the splitting field $\mathbb{C}(x,y)[X]/\langle m(X) \rangle$ into a field of Puiseux series. Our reasoning is based on~\cite{Manuel}, an exposition of a theory of Laurent series in several variables, and on~\cite{MacDonald,buchacher2022newton}, a discussion of a (generalized) Newton-Puiseux algorithm. For details, in particular for proofs, we refer to these references. 

We recall some definitions from convex geometry before we introduce the series we will work with.

\begin{Definition}
A subset $C$ of $\mathbb{R}^n$ is called a cone if $\lambda C = C$ for every $\lambda\in\mathbb{R}_{\geq 0}$. It is called a polyhedral cone if there are $v_1,\dots,v_k \in \mathbb{R}^n$ such that $C = \mathrm{cone}\{v_1,\dots, v_k\} := \mathbb{R}_{\geq 0} v_1 + \dots + \mathbb{R}_{\geq 0} v_k$, and rational if $v_1,\dots,v_k$ can be chosen to be elements of $\mathbb{Q}^n$. A cone $C$ is called convex if $\lambda v + (1-\lambda) w\in C$ for all $v,w\in C$ and all $\lambda\in[0,1]$, and strictly convex if, in addition, $C\cap (-C) = \{0\}$. The dual $C^*$ of $C$ is $C^*=\{ u\in\mathbb{R}^n \mid  u \cdot C \leq 0 \}$.
\end{Definition}

A \emph{series} $\phi$ in $\bold{x} = (x_1,\dots,x_n)$ over~$\mathbb{C}$ is a formal sum 
\begin{equation*}
\phi = \sum_{I\in\mathbb{Q}} a_I \bold{x}^I 
\end{equation*}
of terms in $\bold{x}$ whose coefficients $a_I$ are elements of $\mathbb{C}$.
Its \emph{support} is defined by 
\begin{equation*}
\mathrm{supp}(\phi) = \{I\in\mathbb{Q}^n: a_I \neq 0 \},
\end{equation*}
and we will assume throughout that there is a vector~$v\in\mathbb{R}^n$, a strictly convex rational cone $C\subseteq\mathbb{R}^n$ and an integer $k\in\mathbb{Z}$ such that 
\begin{equation*}
\mathrm{supp}(\phi) \subseteq \left( v + C \right) \cap \frac{1}{k}\mathbb{Z}^n.
\end{equation*}
 The convex hull of $\mathrm{supp}(\phi)$ is called the \emph{Newton polyhedron} of $\phi$. It is denoted by $\mathrm{Newt}(\phi)$. 
 
Given an additive total order $\preceq$ on $\mathbb{Q}^n$, we denote by $\mathbb{C}_{\preceq}((\bold{x}))$ the set of series whose support have a maximal element with respect to $\preceq$. The proof of~\cite[Thm 15]{Manuel} shows that $\mathbb{C}_{\preceq}((\bold{x}))$ is a field, and by~\cite{MacDonald} it is algebraically closed.

The next theorem~\cite[Thm~4]{robbiano1985} gives a characterization of additive total orders on $\mathbb{Q}^n$.

\begin{Definition}
Let $w\in\mathbb{R}^n$. The rational dimension of $w$, denoted by $\mathrm{d}(w)$, is the dimension of the $\mathbb{Q}$-vector space generated by the components of $w$.
\end{Definition}
\begin{Theorem}
For any additive total order $\preceq$ on $\mathbb{Q}^n$, there exist non-zero pairwise orthogonal vectors $u_1,\dots,u_s\in\mathbb{R}^n$ such that $\mathrm{d}(u_1) + \dots + \mathrm{d}(u_s) = n$ and
\begin{equation*}
\iota: (\mathbb{Q}^n, \preceq) \rightarrow (\mathbb{R}^s, \preceq_{lex}) \quad \text{defined by} \quad \iota(v) = (v\cdot u_1,\dots, v\cdot u_s)
\end{equation*}
is an injective order homomorphism.
\end{Theorem}

For simplicity, but without loss of generality, we restrict ourselves to total orders $\preceq$ on $\mathbb{Q}^n$ that are induced by some~$w\in\mathbb{R}^n$ such that
\begin{equation*}
\alpha \preceq \beta \quad \Longleftrightarrow \quad  \alpha \cdot w \leq \beta \cdot w.
\end{equation*}

Having chosen a total order $\preceq$ on $\mathbb{Q}^2$ we can identify the field of rational functions~$\mathbb{C}(x,y)$ with a subfield of $\mathbb{C}_{\preceq}((x,y))$. The series in $\mathbb{C}_{\preceq}((x,y))$ associated with a rational function $p/q\in\mathbb{C}(x,y)$ is 
\begin{equation}\label{eq:rat}
\frac{p}{\mathrm{lt}_{\preceq}(q)} \sum_{k \geq 0} \left( 1- \frac{q}{\mathrm{lt}_{\preceq}(q)}\right)^k.
\end{equation}
Viewing $\mathbb{C}(x,y)$ as a subfield of~$\mathbb{C}_{\preceq}((x,y))$, any series root $\phi\in\mathbb{C}_{\preceq}((x,y))$ of $m(X)$ induces an embedding of $\mathbb{C}(x,y)[X]/ \langle m(X) \rangle$ into $\mathbb{C}_{\preceq}((x,y))$ via
\begin{align*}
p(X) + \langle m(X) \rangle \quad \mapsto \quad  p(\phi).
\end{align*}
This embedding allows us to study equation~\eqref{eq:orbitEquation} in the form
\begin{equation}\label{eq:orbitEquation1}
F(x,y) + \sum_{(p_1,p_2,p_3)} p_3(\phi) F(p_1(\phi),p_2(\phi)) = p(\phi),
\end{equation}
which involves only series to which $[x^\geq y^\geq]$ can be applied. The question how such series roots, and hence such embeddings, can be constructed is answered by the Newton-Puiseux algorithm. We only state the specification of the algorithm here and refer to~\cite{MacDonald,buchacher2022newton} for details. We first recall several definitions. 

Let $p\in\mathbb{C}[\bold{x},y]$, and let $e$ be an edge of $\mathrm{Newt}(p)$ that connects two vertices $v_1$ and $v_2$. The edge is called \emph{admissible}, if $v_{1,n+1}\neq v_{2,n+1}$. If $v_{1,n+1} < v_{2,n+1}$, we call $v_1$ and $v_2$ the minor and major vertex of~$e$, respectively, and denote them by~$\mathrm{m}(e)$ and $\mathrm{M}(e)$. Let $P_e$ be the projection on $\mathbb{R}^{n+1}$ that projects on~$\mathbb{R}^n\times \{0\}$ along lines parallel to an (admissible) edge $e$. The \emph{barrier cone} of~$e$ is the smallest cone that contains $P_e(\mathrm{Newt}(p)) - P_e(e)$, the projection of the Newton polytope of $p$ shifted by the projection of its edge $e$. It is denoted by $C(e)$. We occasionally identify $\mathbb{R}^n$ and $\mathbb{R}^n\times \{0\}\subseteq \mathbb{R}^{n+1}$, and consider $C(e)$ as a subset of $\mathbb{R}^n$. A vector $w\in\mathbb{R}^n$ that defines a total order on $\mathbb{Q}^n$ is said to be \emph{compatible} with $e$, if $w\in C(e)^*$. The edge polynomial $p_e(t)$ of an edge $e$ of the Newton polytope of $p$ is 
\begin{equation*}
p_e(t)=\sum_{I} a_I t^{I_{n+1}-\mathrm{m}(e)_{n+1}},
\end{equation*}
where $a_I = [(\bold{x},y)^I]p$ and the sum runs over all $I$ in~$e\cap \mathrm{supp}(p)$.

\begin{Algorithm}[Newton-Puiseux Algorithm]\label{alg:NPA}
Input: a square-free and non-constant polynomial $p\in\mathbb{C}[\bold{x},y]$, an admissible edge $e$ of its Newton polytope, an element $w$ of the dual of its barrier cone $C(e)$ inducing a total order on $\mathbb{Q}^n$, and a (non-negative) integer $k$.\\
  Output: a list of $\mathrm{M}(e)_{n+1}-\mathrm{m}(e)_{n+1}$ many pairs $(c_1\bold{x}^{\alpha_1}+\dots+c_N \bold{x}^{\alpha_N},C)$ with $c_1\bold{x}^{\alpha_1},\dots,c_N\bold{x}^{\alpha_N}$ being the first $N$ terms of a series solution $\phi$ of $p(\bold{x},\phi) = 0$, ordered with respect to $w$ in decreasing order, and $C$ being a strictly convex rational cone such that $\mathrm{supp}(\phi)\subseteq \{\alpha_1,\dots,\alpha_{N-1}\} \cup \left(\alpha_N + C\right)$, where $N$ is the smallest integer greater than or equal to $k$ such that the series solutions can be distinguished by their first $N$ terms.
\end{Algorithm}

For the purpose of this paper, we only need to know \emph{what} the algorithm does and not \emph{how} it does it, and therefore refer to~\cite{MacDonald,buchacher2022newton} for the steps of the algorithm. The following proposition is a consequence of the correctness of Algorithm~\ref{alg:NPA}. See~\cite[Prop~1]{buchacher2022newton} for a proof, and~\cite[Corollary~4.1]{MacDonald} for a similar statement.
\begin{Proposition}\label{prop:algClosure}
Let $p\in\mathbb{C}[\bold{x},y]$, and let $w\in\mathbb{R}^n$ define a total order $\preceq$ on $\mathbb{Q}^n$. Then $\mathbb{K}_{\preceq}((\bold{x}))$ contains $\deg_y (p)$ many series roots of $p$ all of which can be computed by Algorithm~\ref{alg:NPA}.
\end{Proposition}

Whether the orbit-sum method concludes, and hence $F(x,y) = [x^\geq y^\geq] p(\phi)$, may depend on the total order $\preceq$ and the series root $\phi$ of $m(X)$ in $\mathbb{C}_\preceq((x,y))$, that is, the embedding of $\mathbb{C}(x,y)[X] / \langle m(X) \rangle$ into a field of Puiseux series. As there are infinitely many different fields of Puiseux series, a priori there are infinitely many different embeddings to consider. However, the number of embeddings can be narrowed down as we are not interested in the embeddings themselves but only in the series given elements of $\mathbb{C}(x,y)[X] / \langle m(X) \rangle$ are mapped to.

\begin{Example}\label{ex:rational}
Although there are infinitely many different fields of Puiseux series the multiplicative inverse of $q := 1+x+y$ can be considered an element of, there are only three series it can be identified with. As pointed out before, the corresponding series in $\mathbb{C}_\preceq((x,y))$ is $\frac{1}{\mathrm{lt}_\preceq(q)}\sum_k \left( 1 - q/\mathrm{lt}_\preceq(q) \right)^k$, and depending on what $\mathrm{lt}_\preceq(q)$ is, it is either 
\begin{equation*}
\sum_k (-1)^k(x+y)^k, \quad \overline{x} \sum_k (-1)^k (\overline{x}y + \overline{x})^k \quad \text{or} \quad \overline{y} \sum_k (-1)^k (x\overline{y} + \overline{y})^k.
\end{equation*}
\end{Example}
As a consequence of these observations, we do not work with single total orders but families thereof. 

\begin{Definition}
Let $\phi$ be a series. The smallest rational cone containing the elements of $\mathbb{R}^n$ that induce a total order on $\mathbb{Q}^n$ for which $\mathrm{supp}(\phi)$ has a maximal element is called the order cone of $\phi$. It is denoted by $C(\phi)$.
\end{Definition}

The following proposition is immediate.
\begin{Proposition}
Let $\preceq$ be a total order on $\mathbb{Q}^n$, and let $w\in\mathbb{R}^n$ be a vector it is induced by. Then 
\begin{equation*}
\phi \in\mathbb{C}_\preceq((\bold{x})) \quad \Longleftrightarrow \quad w\in C(\phi).
\end{equation*} 
\end{Proposition}

\begin{Definition}
Let $P\subseteq \mathbb{R}^n$. The recession cone of $P$ is 
\begin{equation*}
\mathrm{recCone}(P) := \left \{  v\in \mathbb{R}^n : P + \mathbb{R}_{\geq 0} v \subseteq P \right\}
\end{equation*}
\end{Definition}

\begin{Proposition}\label{prop:orderCone}
Let $\phi$ be a series. Then 
\begin{equation*}
C(\phi) = \mathrm{recCone}(\mathrm{Newt}(\phi))^*. 
\end{equation*}
\end{Proposition}
\begin{proof}
Any polyhedral set $P$ is the Minkowski sum of the convex hull of its vertices and its recession cone~\cite[Thm~1.2]{ziegler2012lectures}. Hence a linear functional $w$ has a maximum on $P$ if and only if it has one on its recession cone.
\end{proof}

The computation of $C(\phi)$ requires to compute the unbounded faces $\mathrm{Newt}(\varphi)$. For series that are algebraic over $\mathbb{C}(\bold{x})$, this is still an open problem~\cite[Prob~2]
{buchacher2022newton}. However, for rational series this is easy.

\begin{Definition}
Let $P\subseteq\mathbb{R}^n$ be a polyhedron, and let $v\in P$ be a vertex. The smallest cone $C\subseteq\mathbb{R}^n$ such that 
\begin{equation*}
P \subseteq v + C 
\end{equation*}
is called the vertex cone of $P$ at $v$. It is denoted by $C(P,v)$.
\end{Definition}

\begin{Corollary}
Let $p/q\in\mathbb{C}(\bold{x})$, let $\preceq$ be a total order on $\mathbb{Q}^n$, and let $\phi$ be the series in $\mathbb{C}_\preceq((\bold{x}))$ associated with $p/q$. Then 
\begin{equation*}
C(\phi) = C(\mathrm{Newt}(q),\mathrm{lexp}_\preceq(q))^*.
\end{equation*}
\end{Corollary}
\begin{proof}
The explicit expression of $\phi$ given in~\eqref{eq:rat} shows that
\begin{equation*}
\mathrm{recCone}\left(\mathrm{Newt}(\phi)\right) = C(\mathrm{Newt}(q),\mathrm{lexp}_\preceq(q)). 
\end{equation*} 
So the statement follows from Proposition~\ref{prop:orderCone}.   
\end{proof}

\begin{Example}
We continue with Example~\ref{ex:rational}. Depending on the series the rational function $p/q$ is identified with, a total order on $\mathbb{Q}^2$ can be associated with one of three sets: the duals of the vertex cones of the Newton polytope of $q$, that is, the duals of 
\begin{equation*}
\mathrm{cone}\{(1,0), (0,1)\}, \quad \mathrm{cone}\{(-1,1), (-1,0)\} \quad \text{and} \quad \mathrm{cone}\{(1,-1), (0,-1)\}.
\end{equation*}  
\end{Example}

\begin{Definition}
Two series $\phi$, $\psi$ are said to be compatible if there is a field $\mathbb{C}_\preceq((\bold{x}))$ of Puiseux series both are elements of.
\end{Definition}

\begin{Proposition}
Two series $\phi$ and $\psi$ are compatible if and only if $\mathrm{int}\left( C(\phi) \cap C(\psi)\right)$, the interior of the intersection of their order cones, is non-empty. Furthermore, $\mathrm{int}\left(C(\phi) \cap C(\psi) \right) \neq \emptyset$ if and only if $\mathrm{recCone}(\mathrm{Newt}(\phi)) + \mathrm{recCone}(\mathrm{Newt}(\psi))$ is strictly convex.
\end{Proposition}
\begin{proof}
The first part of the statement is an immediate consequence of the definitions. The other part follows from Proposition~\ref{prop:orderCone}, the fact that for two cones $C_1$ and $C_2$ we have 
\begin{equation*}
C_1^* \cap C_2^* = \left( C_1 + C_2\right)^*,
\end{equation*} 
and the observation that $(C_1+C_2)^*$ has full dimension if and only if $C_1+C_2$ is strictly convex.
\end{proof}

Given a finite set of rational functions one can wonder about the different series expansions that are compatible with each other. It turns out that there is a simple description in terms of normal fans.

\begin{Definition}
A polyhedral complex $\Sigma$ is a set of polyhedra satisfying the following two properties: If $P\in\Sigma$ and $F$ is a face of $P$, then $F\in\Sigma$. If $P, Q\in\Sigma$, then $P\cap Q$ is either the empty set, or a face of $P$ and $Q$.
\end{Definition}

\begin{Definition}
Let $P$ be a polyhedron, and let $F$ be one of its faces. The normal cone of $F$ is 
\begin{equation*}
C_F = \{ w\in\mathbb{R}^n : F\subseteq  \underset{x\in P}{\operatorname{argMax}} \hspace{4pt} w\cdot x \}.
\end{equation*}
The normal fan of $P$ is the polyhedral fan whose elements are the normal cones of the faces of $P$.
\end{Definition}

\begin{Definition}
Let $\Sigma_1$ and $\Sigma_2$ be two polyhedral complexes. Their common refinement is 
\begin{equation*}
\Sigma_1 \wedge \Sigma_2 = \left\{ C\cap C' : C\in \Sigma_1, C'\in\Sigma_2 \right\}.
\end{equation*}
\end{Definition}

\begin{Proposition}
There is a bijection between the compatible series expansions of $p_1/q_1, \dots, p_n/q_n \in\mathbb{C}(\bold{x})$ and the $n$-dimensional faces of the common refinement of the normal fans of $\mathrm{Newt}(q_1), \dots, \mathrm{Newt}(q_n)$.
\end{Proposition}
\begin{proof}
There is a bijection between the series expansions of a rational function $p/q$ and the vertices of $\mathrm{Newt}(q)$. Their order cones are given by the duals of the corresponding vertex cones. They form the $n$-dimensional faces of the normal fan of $\mathrm{Newt}(q)$. Series expansions $\phi_1,\dots, \phi_n$ of $p_1/q_1, \dots, p_n/q_n$ are compatible if and only if the interior of $C(\phi_1)\cap \dots \cap C(\phi_n)$ is non-empty. This is the case if and only if $C(\phi_1)\cap \dots \cap C(\phi_n)$ is an $n$-dimensional face of the common refinement of the normal fans associated with the Newton polytopes of $q_1,\dots,q_n$.
\end{proof}

\begin{Example}
We explain how to compute the compatible series expansion of $1/(1+x+y)$ and $1/(x+y+xy)$. Each rational function has three series expansions. However, not all of them are compatible. To determine those which are we consider the normal fans of the Newton polygons of $1+x+y$ and $x+y+xy$ and their common refinement (see the figure above). The latter consists of six $2$-dimensional cones each of which corresponds to a pair of compatible series expansions. To determine the one associated with $\mathrm{cone}\{(-1,-1), (0,-1)\}$, for instance, we choose any of its elements that induces a total order $\preceq$, and compute the series expansions in $\mathbb{C}_\preceq((x,y))$:
\begin{equation*}
\sum_k (-1)^k (x+y)^k \quad \text{and} \quad x^{-1} \sum_k (-1)^k (x^{-1}y+y)^k.
\end{equation*}

\end{Example}

\begin{figure}
\begin{center}
  \begin{tikzpicture}[scale=.2]
    \fill[lightgray] (-2,-2)--(4,-2)--(-2,4)--cycle;
    \draw[-](0,0)--(5,5);
    \draw[-](0,0)--(0,-5);
    \draw[-](0,0)--(-5,0);
 \begin{scope}[xshift=20cm]
   \fill[lightgray] (2,2)--(-4,2)--(2,-4)--cycle;
    \draw[-](0,0)--(-5,-5);
    \draw[-](0,0)--(0,5);
    \draw[-](0,0)--(5,0);
    \end{scope}
     \begin{scope}[xshift=40cm]
    \draw[-](0,0)--(5,5);
    \draw[-](0,0)--(0,-5);
    \draw[-](0,0)--(-5,0);
    \draw[-](0,0)--(-5,-5);
    \draw[-](0,0)--(0,5);
    \draw[-](0,0)--(5,0);
    \end{scope}
  \end{tikzpicture}
\end{center}
  \caption{The normal fans of the Newton polygons of $1+x+y$ and $x+y+xy$ and their common refinement.}
\end{figure}

To determine all series interpretations of equation~\eqref{eq:orbitEquation}, we need to clarify how to compute all series solutions of $m(X) = 0$. Given a total order $\preceq$, we can compute all solutions that are elements of $\mathbb{C}_\preceq((x,y))$. However, these may not be all. There might be more as $\preceq$ ranges over the total orders of $\mathbb{Q}^2$. It is therefore necessary to ask to which extent the output of Algorithm~\ref{alg:NPA} depends on the total order of the input. In turns out that if the edge polynomial of $m(X)$ associated with an edge $e$ is square-free, which is the case for any generic polynomial, then the series roots constructed from $e$ do not depend on the total order at all~\cite[Prop~6]{buchacher2022newton}. In particular, if all edge polynomials of $m(X)$ are square-free, then there are only finitely many series solutions. Furthermore, total orders that give rise to them can be read off from the barrier cones of the corresponding edges. If an edge polynomial is not square-free, then it is still an open question whether there are only finitely many series solutions, and how total orders that give rise to them can be found~\cite[Conj~2]{buchacher2022newton}. 

In order to effectively apply $[x^\geq y^\geq]$ to equation~\eqref{eq:orbitEquation1}, we need to bound the support of 
\begin{equation*}
\sum_{(p_1,p_2,p_3)} p_3(\phi) F(p_1(\phi),p_2(\phi)).
\end{equation*}
We do so by bounding the supports of its summands, since if $[x^\geq y^\geq] p_3(\phi) F(p_1(\phi),p_2(\phi)) = 0$ for all~$(p_1,p_2,p_3)$, then $F(x,y) = [x^\geq y^\geq] p(\phi)$. The derivation of the bounds will again rely on the Newton-Puiseux algorithm, and on Theorem~\ref{theorem:comp} below. We recall that the Newton-Puiseux algorithm is not only useful for constructing series solutions of polynomial equations but also for deriving information about the convex hull of their supports~\cite[Sec~5]{buchacher2022newton}. In particular, for each $p_i(\phi)$ in equation~\eqref{eq:orbitEquation1} we can compute the (finitely many) vertices of the convex hull of its support, and for each of these vertices $v$ we can determine a (strictly convex) cone $C$ such that $\mathrm{supp}(p_i(\phi))\subseteq v + C$. Again, in general it is an open problem~\cite[Prob~1]{buchacher2022newton} how to find the vertex cones, that is, the cones that are minimal. 

The Newton polyhedron of the product of two series is a subset of the Minkowski sum of their Newton polyhedra. In order to compute an estimate of the support of $p_3(\phi) F(p_1(\phi),p_2(\phi))$, it remains to clarify how to derive an estimate for the support of $F(p_1(\phi),p_2(\phi))$. The following theorem~\cite[Thm~17]{Manuel} gives a sufficient condition for the composition of Puiseux series to be well-defined, and in case it is, it provides a cone that contains its support. 

\begin{Theorem}\label{theorem:comp}
Let $C\subseteq\mathbb{R}^n$ be a strictly convex cone and $F(x_1,\dots,x_n)$ a series such that $\mathrm{supp}(F)\subseteq C$, let~$\preceq$ be an additive order on $\mathbb{Z}^m$ and $g_1,\dots,g_n\in\mathbb{C}_{\preceq}((y_1,\dots,y_m))\setminus \{0\}$. Furthermore, let $M\in\mathbb{Z}^{m\times n}$ be the matrix whose $i$-th column consists of the leading exponent of~$g_i(y_1,\dots,y_m)$ with respect to $\preceq$, and let $C'$ be a cone that contains the image of $C$ under~$M$ and $\mathrm{supp}(g_i / \mathrm{lt}(g_i))$ for $i=1,\dots,n$. If $C \cap \mathrm{ker}(M) = \{0\}$ and if $C'$ is strictly convex, then~$F(g_1,\dots,g_n)$ is well-defined and $\mathrm{supp}(F(g_1,\dots,g_n))\subseteq C'$.
\end{Theorem}

We are interested in applying Theorem~\ref{theorem:comp} when $F\in\mathbb{C}[x,y][[t]]$ satisfies $[t^0] F = 1$, and $g_1,g_2\in\mathbb{C}_{\preceq}((x,y))$ and $g_3 = t$. In this case the assumptions of the theorem are always fulfilled.

\begin{Corollary}\label{lemma:1}
Let $C$ be a strictly convex cone in $\mathbb{R}^3$ such that $C\cap \left(\mathbb{R}^2\times \{0\}\right) = \{0\}$, and let $F\in\mathbb{C}[x,y][[t]]$ be such that $\mathrm{supp}(F)\subseteq C$. Let $\preceq$ be an additive total order on $\mathbb{Q}^3$ and $g_1,g_2\in\mathbb{C}_{\preceq}((x,y))$, and let $M$ be the matrix whose columns are the leading exponents of $g_1,g_2$ and $t$. Then $C\cap \mathrm{ker}(M) = \{0\}$, and the cone generated by $MC$ and $\mathrm{supp}(g_i/\mathrm{lt}(g_i))$ for~$i\in\{1,2\}$ is strictly convex.
\end{Corollary}
\begin{proof}
The series $g_1$ and $g_2$ do not depend on $t$, therefore $\mathrm{ker}(M) \subseteq \mathbb{R}^2\times \{0\}$, and so~$C\cap \mathrm{ker}(M) = \{0\}$, by assumption on $C$. Since $g_1$ and $g_2$ are elements of $\mathbb{C}_{\preceq}((x,y))$, the cone generated by the support of $g_1/\mathrm{lt}(g_1)$ and $g_2/\mathrm{lt}(g_2)$ is strictly convex, and because $g_1$ and $g_2$ are independent of $t$, it is contained in $\mathbb{R}^2\times \{0\}$. The shape of $M$ implies $MC \cap\left( \mathbb{R}^2\times \{0\}\right) = M\left(C \cap\left( \mathbb{R}^2\times \{0\}\right) \right) =  \{0\}$. To finish the proof of the lemma, it is therefore sufficient to show that $MC$ is strictly convex. Assume that there is a~$v\neq 0$ such that $v\in MC$ and $-v\in MC$. Then there are $u_1,u_2\in C$ such that~$Mu_1=v$ and $Mu_2=-v$. But then $M(u_1+u_2) = 0$, i.e. $u_1+u_2\in\mathrm{ker}(M)$. Together with~$u_1+u_2\in C$ and $C\cap\mathrm{ker}(M) = \{0\}$ this implies that $u_1 + u_2 = 0$. Since~$C$ is strictly convex, $u_1,u_2=0$, and therefore $v=0$. So $MC$ is strictly convex as well.
\end{proof}

To summarize, we can construct series roots $\phi$ of $m(X)$ to embed $\mathbb{C}(x,y)[X]/\langle m(X) \rangle$ into fields $\mathbb{C}_{\preceq}((x,y))$ of Puiseux series and interpret equation~\eqref{eq:orbitEquation} as an identity of series. For each such embedding we can determine the vertices of the convex hull of the support of $p_i(\phi)$, and for each of these vertices $v$ we can compute a strictly convex rational cone $C_v$ such that $\mathrm{supp}(p_i(\phi))\subseteq v + C_v$. Furthermore, we are able to find a strictly convex cone $C$ such that $\mathrm{supp}(F(p_1(\phi),p_2(\phi))) \subseteq C$. The support of $p_3(\phi) F(p_1(\phi),p_2(\phi))$ is then contained in $v + C_v + C$. If $\left( \mathbb{Q}_{\geq 0}^2 \times \mathbb{Q}\right) \cap \left( v + C_v+C\right) = \emptyset$, then $[x^\geq y^\geq] p_3(\phi) F(p_1(\phi),p_2(\phi)) = 0$.

\begin{Example}\label{ex:system1}
We solve the system of discrete differential equations 
\begin{align}\label{eq:system}
\begin{split}
F_0 &= 1 + t F_1 + t \Delta_x \Delta_y F_1\\
F_1 &= t(1+x+y)F_0 + ty \Delta_x F_0
\end{split}
\end{align}
for $F_0,F_1\in\mathbb{C}[x,y][[t]]$ and show that their solution is D-finite. We begin with eliminating~$F_1(x,y)$ from the first of these equations and continue working with
\begin{equation}\label{eq:kernelEqu}
(1-t^2S_0 S_1) F_0 = 1 - t\bar{x}\bar{y} (F_1(x,0) + F_1(0,y) - F_1(0,0)) - t^2(\bar{x}\bar{y}+1)\bar{x}y F_0(0,y),
\end{equation}
where 
\begin{equation*}
S_0 := \bar{x}y+y+x+1 \quad \text{and} \quad S_1 := \bar{x}\bar{y} + 1.
\end{equation*}
The Laurent polynomial $S_0S_1$ has a finite orbit. Its elements are $(x,y),(x,\bar{y})$, $(p_1(\alpha),y)$ and $(p_{-1}(\alpha),y)$, and $(p_1(\alpha),\bar{y})$ and $(p_{-1}(\alpha),\bar{y})$, where
\begin{equation*}
p_i(X) = \frac{x + y + x y + x y^2 + i X}{2 x^2 y} \quad \text{and} \quad \alpha = \sqrt{4 x^3 y^2 + (x + y + x y + x y^2)^2}.
\end{equation*}
We consider their components as elements of the extension of $\mathbb{C}(x,y)$ by a root $\alpha$ of
\begin{equation*}
m(X) = X^2 - 4 x^3 y^2 - (x + y + x y + x y^2)^2.
\end{equation*}
Plugging the elements of the orbit into equation~\eqref{eq:kernelEqu},
forming a linear combination of the resulting equations with undetermined coefficients, and equating the  coefficients of the sections of~$F_0$ and $F_1$ to zero results in a linear system over $\mathbb{C}(x,y)[\alpha]$. The vector space of solutions is $1$-dimensional, and so is the vector space of section-free orbit equations. The latter is generated by the equation
\begin{equation*}
F_0(x,y) - \bar{y}^2 F_0(x,\bar{y}) - \sum_{i,j = \pm1} c_{ij}(\alpha) F_0(p_i(\alpha),y^j) = \frac{(-1 + y^2) (2 y - x^3 y + x (1 + y + y^2))}{x^3 y^3(1 - t^2 S_0 S_1)}.
\end{equation*}
The coefficients $c_{ij}(\alpha)$ in $\mathbb{C}(x,y)[\alpha]$ are
\begin{align*}
c_{ij}(\alpha) = i \frac{x + 2 y + x y + x y^2}{2 x^3 y} + j \frac{\alpha (2 y^2 + 2 x^3 y^2 + 3 x y (1 + y + y^2) + x^2 (1 + y + y^2)^2)}{2 x^3 y (y^2 + 4 x^3 y^2 + 2 x y (1 + y + y^2) + x^2 (1 + y + y^2)^2)}.
\end{align*}
Let $\preceq$ be the total order on $\mathbb{Q}^2$ defined by $w = (\sqrt{2},1/2)$, and let $\phi$ be the series solution of $m(X) = 0$ in $\mathbb{C}_{\preceq}((x,y))$ whose first term is $2x^{3/2}y$.
We identify $p_i(\alpha)$ and~$c_{ij}(\alpha)$ with $p_i(\phi)$ and $c_{ij}(\phi)$ in $\mathbb{C}_{\preceq}((x,y))$ and show that the only term on the left hand side of the equation that remains when applying $[x^{\geq}y^{\geq}]$ is $F_0(x,y)$. Consequently, $F_0$ is the non-negative part of a rational function, and therefore D-finite. By the second of the equations in~\eqref{eq:system} so is then $F_1$. Obviously,~$[x^{\geq}y^{\geq}] F_0(x,y) = F_0(x,y)$ and $[x^{\geq}y^{\geq}] \bar{y}^2F_0(x,\bar{y}) = 0$. Using the Newton-Puiseux algorithm, one can show that 
\begin{equation*}
\mathrm{supp}(p_i(\phi)) \subseteq (-1/2,0,0) + \langle (-1,2,0),(-1,-2,0) \rangle
\end{equation*}
and 
\begin{equation*}
\mathrm{supp}(c_{ij}(\phi)) \subseteq (-3/2,-1+j,0) + \langle (-1,2,0),(-1,-2,0) \rangle.
\end{equation*}
Theorem~\ref{theorem:comp} then implies that
\begin{equation*}
\mathrm{supp}(F_0(p_i(\phi),y^j))\subseteq \langle (0,0,1),(0,j,1),(-1,2,0),(-1,-2,0) \rangle.
\end{equation*}
Therefore, 
\begin{equation*}
\mathrm{supp}(c_{ij}(\phi)F_0(p_i(\phi),y^j)) \subseteq (-3/2,-1+j,0) + \langle (0,0,1),(0,j,1),(-1,2,0),(-1,-2,0) \rangle,
\end{equation*}
and so 
\begin{equation*}
[x^{\geq}y^{\geq}] c_{ij}(\phi) F_0(p_i(\phi),y^j) = 0.
\end{equation*}
\end{Example}

The following algorithm gives an effective sufficient condition for the application of $[x^\geq y^\geq]$ to the orbit-sum~\eqref{eq:orbitEquation} to result in an expression of $F$ as the non-negative part of an algebraic function. 

\begin{Algorithm}\label{alg:ppe}
  Input: a series root $\phi$ of $m(X)$ and a cone $C_0$ that contains $\mathrm{recCone}(\mathrm{Newt}(\phi))$, a cone $C_1$ that contains the support of a series $F\in\mathbb{C}[x,y][[t]]$ such that $C_1\cap\left(\mathbb{R}^2\times\{0\}\right) = \{0\}$ and a method for computing the coefficient of any term of~$F$, and a list $L_0$ of tuples $(p_1(X),p_2(X),p_3(X))$ of polynomials over $\set Q(x,y)$.
  
Output: True or Failed, with the output being True only if there is a total order $\preceq$ such that the non-negative part of $p_3(\phi) F(p_1(\phi),p_2(\phi))$ in $\set C_{\preceq}((x,y))[[t]]$ is zero for all $(p_1(X),p_2(X),p_3(X))$ in~$L_0$.
  
 \step 10 Determine the maximal list $L_1$ of minimal cones $C$ such that for every polynomial $p(X)$ which appears as a component of an element of $L_0$ its series expansion in $\mathbb{C}_{\preceq}(x,y)[X]$ does only depend on the cone $C$ but not on the specific total order $\preceq$ induced by an element of~$C^*$.
 \step 20 For each $C\in L_1$ such that $C + C_0$ is strictly convex, do:
 \step 31 Choose any total order $\preceq$ on $\mathbb{Q}^2$ induced by some element of $(C+C_0)^*$, and determine for each $p(X)$ which appears as a component of an element of $L_0$ a list $L_p$ of pairs $(v_p,C_{v_p})$ such that $v_p$ is a vertex of the convex hull of the support of $p(\phi)$ in $\mathbb{C}_{\preceq}((x,y))$ and $C_{v_p}$ is an estimate of the corresponding vertex cone.
 \step 41 If for each $(p_1(X),p_2(X),p_3(X))$ in $L_0$ there are $(v_{p_i},C_{v_{p_i}})$ in $L_{p_i}$ such that for the cone $C'$ computed from $C_1$ and $(v_{p_1},C_{v_{p_1}})$ and $(v_{p_2},C_{v_{p_2}})$ using Theorem~\ref{theorem:comp}, the set
 \begin{equation*}
   \left( \mathbb{Q}_{\geq 0}^2\times \mathbb{Q}\right) \cap \left(v_{p_3} + C_{v_{p_3}} + C' \right)
 \end{equation*}
 is bounded, and for each of its elements $(i,j,n)$, the coefficient of $x^i y^j t^n$ in $p_3(\phi)F(p_1(\phi),p_2(\phi))$
 is zero, then return True.
  \step 50 Return Failed. 
\end{Algorithm}

\begin{Theorem}
  Alg.~\ref{alg:ppe} is correct in the sense that if it returns True,
  then there is a total order $\preceq$ such that $[x^{\geq} y^{\geq}]p_3(\phi) F(p_1(\phi),p_2(\phi)) = 0$ in $\set C_{\preceq}((x,y))[[t]]$ for all $(p_1(X),p_2(X),p_3(X))$ in~$L_0$.
\end{Theorem}
\begin{proof}
  Suppose the algorithm returns True, and let $\preceq$ be the total order chosen in step~3 in
  the iteration when the algorithm terminated.
  If $w\in (C+C_0)^*$ is a weight vector associated with~$\preceq$, then $w\in C_0^*$,
  and therefore $\phi\in \set C_{\preceq}((x,y))$.
  By Corollary~\ref{lemma:1}, it also follows that $p_3(\phi)F(p_1(\phi),p_2(\phi))\in\set C_{\preceq}((x,y))[[t]]$
  for every $(p_1(X),p_2(X),p_3(X))\in L_0$.
  It remains to show that the non-negative parts of these series are all zero.

  For every $(p_1(X),p_2(X),p_3(X))$ in $L_0$,
  the support of $F(p_1(\phi),p_2(\phi))$ is restricted to the~$C'$ considered in step~4.
  The support of $p_3(\phi)$ is restricted to $v_{p_3}+C_{v_{p_3}}$ as chosen in step~3.
  So the support of $p_3(\phi)F(p_1(\phi),p_2(\phi))$ is restricted to
  $v_{p_3}+C_{v_{p_3}}+C'$. By the termination condition, the
  coefficients of the terms whose exponent vectors are elements of $\mathbb{Q}_{\geq 0}^2\times \mathbb{Q}$ are zero, which implies that $[x^{\geq} y^{\geq}]p_3(\phi) F(p_1(\phi),p_2(\phi)) = 0$.  
\end{proof}

If the algorithm returns Failed although True would have been a correct output, this may have several
reasons. One possible explanation is that one of the cones provided as input or chosen in step~3 or
computed according to Theorem~\ref{theorem:comp} were not tight. It is not surprising that oversized cones may
cause the set $\left( \mathbb{Q}_{\geq 0}^2\times \mathbb{Q}\right) \cap \left(v_{p_3} + C_{v_{p_3}} + C'\right)$
to be infinite even though tighter choices for the cones may have led to a bounded set. The following example indicates that the algorithm may also return Failed when all cones are chosen optimally.

\begin{Example}\label{ex:tight}
Consider the set 
\begin{equation*}
S := \{(-1,0,0)\} \cup \{ (-1,-1,0) + k\cdot (1,0,1) : k\in\mathbb{N} \}.
\end{equation*}
Though it does not intersect $\mathbb{Q}_{\geq 0}^2\times \mathbb{Q}$, the intersection of its convex hull with $\mathbb{Q}_{\geq 0}^2\times \mathbb{Q}$ is
\begin{equation*}
\mathrm{conv}\{ (-1,0,0) + k\cdot (1,0,1): k\in\mathbb{N}^* \},
\end{equation*}
and hence unbounded.
\end{Example}

Even if we had an algorithm that would return Failed only when True is not a legitimate answer, this would still
not be enough to cover all cases in which the orbit-sum method is applicable.
In~\cite[Prop~24]{large} it was shown that it can happen that there are two terms $p_3(\alpha)F(p_1(\alpha),p_2(\alpha))$ whose expansions involve terms with non-negative powers in $x$ and $y$, although their sum does not. 

\section{Conclusion}
We have extended the applicability of the orbit-sum method for linear DDE's of higher order. However, there remain many equations where the method fails. First, there are equations which do not admit a solution by the orbit-sum method, simply because the shape of the equation does not allow the method to conclude. In some cases, for instance, the orbit is not finite~\cite[Thm~3]{smallSteps}, and in others the orbit is finite but there is no section-free orbit equation (Example~\ref{ex:system2}), and again for others the section-free orbit equations only have a zero orbit-sum~\cite[Sec~4.2]{smallSteps}. 

\begin{Example}\label{ex:system2}~\cite[Sec~6.10, Ex~15]{buchacher2021algorithms}\\
Similar as in Example~\ref{ex:system1} we try to solve the system of discrete differential equations
\begin{align}\label{eq:system2}
\begin{split}
F_0 &= 1 + t F_1 + tx \Delta_y F_1\\
F_1 &= t y F_0 + t \Delta_x F_0\\
\end{split}
\end{align}
for $F_0,F_1\in\mathbb{Q}[x,y][[t]]$. We begin with eliminating~$F_1(x,y;t)$ from the first of these equations and continue working with
\begin{equation}\label{eq:kernelEqu2}
(1-t^2S_0 S_1) xyF_0 = xy - t x^2 F_1(x,0) - t^2(x+y)F_0(0,y),
\end{equation}
where 
\begin{equation*}
S_0 := \bar{x} + y \quad \text{and} \quad S_1 := 1 + x \bar{y}.
\end{equation*}
The Laurent polynomial $S_0S_1$ has a finite orbit. Its elements are $(x,y)$, $(\bar{x},y)$, $(\bar{x},\bar{y})$ and $(x,\bar{y})$. Replacing $(x,y)$ in equation~\eqref{eq:kernelEqu2} by these gives four equations, which cannot be non-trivially linearly combined to cancel the section. The same is true when working with the equation that results from eliminating $F_0$ from the system of equations~\eqref{eq:system2}. Hence the orbit-sum method cannot conclude.
\end{Example}

Second, there are equations for which the orbit-sum method as presented here fails because we were not able to address some of the problems that can arise. For instance, if there is essentially more than one section-free orbit equation it is not clear which of them should be chosen to extract the non-negative part. The method clearly fails, if the estimates of the order cone of $\phi$ and the supports of~$p_3(\phi)F(p_1(\phi),p_2(\phi))$ are too big. How to determine estimates that are tight, and how to deal with examples such as Example~\ref{ex:tight} and~\cite[Prop~24]{large}, are questions that remain open.

\printbibliography

\end{document}